\documentclass[a4,12pt]{article}%

\usepackage{graphicx}
\usepackage{graphics}
\usepackage{xcolor}

\usepackage{ dsfont }

\usepackage{amsmath}
\usepackage{amsfonts}
\usepackage{amssymb}
\usepackage{enumerate}
\newtheorem{theorem}{Theorem}[section]

\newtheorem{definition}[theorem]{Definition}

\newtheorem{proposition}[theorem]{Proposition}

\newenvironment{proof}[1][Proof]{\textbf{#1.} }{\ \rule{0.5em}{0.5em}}

\newcommand{\val}{{\rm val\ }}

\newcommand{\dN}{{\bf N}}

\newcommand{\dR}{{\bf R}}

\newcommand{\calV}{{\cal V}}
\newcommand{\ep}{\varepsilon}

\newcounter{figurecounter}
\begin{document}

\title{Characterizing the Value Functions of Polynomial Games%
\thanks{The authors acknowledge the support of the Israel Science Foundation, grants \#217/17 and \#722/18, and NSFC-ISF Grant \#2510/17.
We thank Francesca Acquistapace, Fabrizio Broglia, Jos\'{e} Galv\'{a}n, Ehud Lehrer, and Omri Solan for useful discussions.}}

\author{Galit Ashkenazi-Golan, Eilon Solan and Anna Zseleva%
\thanks{The School of Mathematical Sciences, Tel Aviv
University, Tel Aviv 6997800, Israel. e-mail:
galit.ashkenazi@gmail.com, eilons@post.tau.ac.il, zseleva.anna@gmail.com.}}

\maketitle

\begin{abstract}
We provide a characterization of the set of real-valued functions
that can be the value function of some polynomial game.
Specifically, we prove that a function $u : \dR \to \dR$ is the value function of some polynomial game
if and only if $u$ is a continuous piecewise rational function.
\end{abstract}


\section{Introduction}

A \emph{polynomial game}%
\footnote{Dresher, Karlin, and Shapley (1950) studied a related
class of games that they termed \emph{polynomial games}; those are
two-player zero-sum strategic-form games in which the action set of each player is $[0,1]$
and the payoff function is polynomial in $x$ and $y$,
where $x \in [0,1]$ is the mixed action chosen by Player~1 and $y \in [0,1]$ is the mixed action chosen by Player~2.}
is a strategic-form game whose payoffs are polynomials in a parameter $Z$.
The value of the two-player zero-sum polynomial game depends on the value of $Z$, and therefore it is a function $u : \dR \to \dR$.
Since the value of a two-player zero-sum strategic-form game with finitely many strategies
is a solution of a set of linear inequalities (see, e.g., Shapley and Snow (1950)),
it follows that the value function $u$ is a continuous and piecewise rational function; that is,
one can divide the real line $\dR$ into finitely many intervals such that the function $u$ is a rational function on each piece.
In this note we show that the converse also holds:
every continuous and piecewise rational function is the value function of some polynomial game.

Characterizing the collection of value functions of a given model has several reasons.
First, the richness of the set of value functions indicates the complexity of the model,
and allows us to compare models that seem unrelated.
Second, each restriction that the value function must satisfy arises from some aspect(s) in the model,
hence increases our understanding of the model.
Third,
sometimes we are given the value for some parameters, and we would like to estimate the value for other parameters.
Once we identify the set of possible value functions of the model,
we know how many data points we need to estimate the value function for new parameters.
Fourth,
sometimes we are given the value function for some parameters,
yet we do not completely know the underlying model.
The characterization of the set of value functions may allow us to rule out possible models
or increase our confidence in a prospective model.

The note is organized as follows.
In Section~\ref{sec:model} we describe the model and main result,
in Section~\ref{sec:proof} we provide the proof of the main result,
and in Section~\ref{sec:discussion} we discuss extensions of the main result.

\section{The Model and the Main Result}
\label{sec:model}

\begin{definition}
A two-player zero-sum \textup{polynomial game} is a tuple $\Gamma = (A, B, G)$, where
\begin{itemize}
    \item $A$ is the finite strategy set of Player~1,
    \item $B$ is the finite strategy set of Player~2,
    \item $G = ((G_{a,b})_{a \in A, b \in B})$ is the payoff matrix of size $|A| \times |B|$, whose entries are polynomials in $Z$.
\end{itemize}
\end{definition}

For every $z \in \dR$, denote by $\Gamma(z) = (A,B,G(z))$ the two-player zero-sum strategic-form game
where the sets of strategies of the two players are $A$ and $B$ respectively, and the payoffs
are $G(z) = (G_{a,b}(z))_{a \in A, b \in B})$, that is, the payoffs in $\Gamma$ evaluated at $Z=z$.
Denote by $u(z):=\text{val}(\Gamma(z))$ the value of the game $\Gamma(z)$.
The function $u : \dR \to \dR$ is the \emph{value function} of the polynomial game $\Gamma$.
Denote by $\mathcal{V}$ the set of all real-valued functions that are value functions of some polynomial game $\Gamma$.

\begin{definition}
A function $u : \dR \to \dR$ is \emph{piecewise rational}
if there are a natural number $K \in \dN$ and a sequence $-\infty = h_1 < h_2 < \cdots < h_K = \infty$,
and rational functions $(\tfrac{Q_k}{R_k})_{k=1}^{K-1}$ such that\footnotemark
\[ u = \sum_{k=1}^{K-1} \mathbf{1}_{(h_k,h_{k+1}]} \tfrac{Q_k}{R_k}. \]
\end{definition}

\footnotetext{To simplify notation, in the equation below, for $k=K-1$ we write $(h_{K-1},h_K]$ and not $(h_{K-1},h_K)$.}

Our main result is the following.
\begin{theorem}
\label{theorem:1}
The set $\mathcal V$ coincides with the set of all continuous and piecewise rational functions from $\dR$ to $\dR$.
\end{theorem}

\section{Proofs}
\label{sec:proof}

This section is devoted to the proof of Theorem~\ref{theorem:1}.
We start with the following simple properties of the set $\mathcal{V}$.

\begin{proposition}\label{prop:minus}
Let $u, w \in \mathcal{V}$. Then
\begin{enumerate}
  \item[(A.1)] $-u \in \mathcal{V}$.
  \item[(A.2)] $cu \in \mathcal V$ for every $c > 0$.
  \item[(A.3)] $u+w \in \mathcal{V}$.
  \item[(A.4)] $\max\{u,w\} \in \mathcal{V}$.
  \item[(A.5)] If there is $\ep > 0$ such that $u(z) \geq \ep$ for every $z \in \dR$, then $\tfrac{1}{u} \in \mathcal V$.
  \item[(A.6)] If there is $z_0 \in \dR$ such that $u(z_0) = w(z_0)$ then the function $v : \dR \to \dR$ that is defined by
  \[ v(z) := \textbf{1}_{z \leq z_0} u(z) + \textbf{1}_{z > z_0} w(z) \]
  is in $\calV$.
  \item[(A.7)] $u \cdot w \in \calV$.
\end{enumerate}
\end{proposition}

\begin{proof}
Since $u,w \in \mathcal{V}$, there exist games $\Gamma_u = (A_u,B_u,G_u)$ and $\Gamma_w = (A_w,B_w,G_w)$ such that $u$ is the value function of $\Gamma_u$
and $w$ is the value function of $\Gamma_w$.

By changing the role of the two players, the value of the game is multiplied by $-1$.
Formally, the value of the game $(B_u,A_u,(-G_u)^T)$ is minus the value of the game $(A_u,B_u,G_u)$.
Statement~(A.1) follows.

By multiplying all entries of the matrix by the same positive constant $c$,
the value of the game is multiplied by $c$.
It follows that Statement~(A.2) holds as well.

To see that Statement~(A.3) holds, suppose that the two players simultaneously play the games $\Gamma_u$ and $\Gamma_w$,
and the payoff is the sum of the payoffs in the two games.
Formally, we consider the game $\Gamma = (A,B,G)$ where $A = A_u\times A_w$,
$B = B_u \times B_w$, and $G_{(a_u,a_w),(b_u,b_w)} = G_{u}(a_u,b_u)+G_{w}(a_w,b_w)$.
The reader can verify that for every $z \in \dR$, $\val(\Gamma(z)) = \val(\Gamma_u(z)) + \val(\Gamma_w(z))$.

To see that Statement~(A.4) holds, suppose again that the two players simultaneously play the games $\Gamma_u$ and $\Gamma_w$,
and, when choosing his strategies in the two games, Player~1 also chooses whether the payoff will be the payoff in $\Gamma_u$ or in $\Gamma_w$.
Formally, we consider the game $\Gamma = (A,B,G)$ where $A = A_u\times A_w \times \{U,W\}$,
$B = B_u \times B_w$, $G_{(a_u,a_w,U),(b_u,b_w)} = G_{u}(a_u,b_u)$, and $G_{(a_u,a_w,W),(b_u,b_w)} = G_{w}(a_w,b_w)$.
The reader can verify that for every $z \in \dR$, $\val(\Gamma(z)) = \max\{\val(\Gamma_u(z)),\val(\Gamma_w(z))\}$.

We turn to prove that Statement~(A.5) holds.
Let $\Gamma = (A,B,G)$ be the game where $A = \{1\} \cup A_u$, $B = \{1\} \cup B_u$, and
    \begin{align*}
    G =
    \begin{Bmatrix}
    \varepsilon & \boldsymbol{0}^{1\times|B_u|} \\
    \boldsymbol{0}^{|A_u|\times1} & G_u-\boldsymbol{\varepsilon}^{|A_u|\times|B_u|}
    \end{Bmatrix} \ ,
    \end{align*}
where $\boldsymbol{0}^{1\times|B_u|}$ and $\boldsymbol{0}^{|A_u|\times1}$ are matrices of sizes ${1\times|B_u|}$ and $|A_u|\times1$
all of whose elements are 0,
and $\boldsymbol{\varepsilon}^{|A_u|\times|B_u|}$ is the matrix of size $|A_u|\times|B_u|$
all of whose elements are $\ep$.
The value of the strategic-form game $(A_u,B_u,G_u(z)-\boldsymbol{\varepsilon}^{|A|\times|B|})$ is $u(z) - \ep$, which is positive.
Since $\varepsilon > 0$, it follows that the value of the strategic-form game $(A,B,G(z))$ is the same as the value of the strategic-form $2\times 2$ game
    \begin{align*}
    \begin{Bmatrix}
    \varepsilon & 0 \\
    0 & u(z)-\varepsilon
    \end{Bmatrix} ,
    \end{align*}
which is $\tfrac{\varepsilon(u(z)-\varepsilon)}{u(z)} = \varepsilon - \tfrac{\varepsilon^2}{u}$.
The result follows from Statements~(A.1)--(A.3).

We next prove that Statement~(A.6) holds.
We note that for every two rational functions $\widehat u,\widehat w : \dR \to \dR$ such that $\widehat u(0) = \widehat w(0) = 0$
there are $c > 0$ and $N \in \dN$ sufficiently large such that
$\widehat u(z) \geq cz + z^{2N+1}$ for every $z \leq 0$,
and $\widehat w(z) \leq cz + z^{2N+1}$ for every $z \geq 0$.
This implies that there are two polynomials $P,Q : \dR \to \dR$ such that
\begin{itemize}
\item
$u(z) \geq P(z)$ for every $z \leq z_0$,
and $u(z),w(z) \leq P(z)$ for every $z \geq z_0$.
\item
$u(z),w(z) \leq Q(z)$ for every $z \leq z_0$,
and $w(z) \geq Q(z)$ for every $z \geq z_0$.
\end{itemize}
Since $v = \min \{ \max\{u,P\}, \max\{w,Q\} \}$,
the claim follows by Statements~(A.1) and~(A.4).

We finally prove that Statement~(A.7) holds.
Since $u,w \in \calV$, both these functions are piecewise rational.
Hence,
there exist a natural number $K \in \dN$ and a sequence $-\infty = h_1 < h_2 < \cdots < h_K = \infty$
such that the sign of both $u$ and $w$ is the same over each interval $(h_k,h_{k+1})$, for $k \in \{1,2,\ldots,K-1\}$.

Let $\Gamma = (A,B,G)$ be the polynomial game in which the two players simultaneously play the games $\Gamma_u$ and $\Gamma_w$,
and the payoff is the product of the payoffs in the two games.
Formally, $A = A_u\times A_w$,
$B = B_u \times B_w$, and $G_{(a_u,a_w),(b_u,b_w)} = G_{u}(a_u,b_u) \cdot G_{w}(a_w,b_w)$.
Since the expectation of the product of two independent random variables is the product of the expectations,
it follows that if $z$ is such that both $u(z) \geq 0$ and $w(z) \geq 0$,
then $\val(\Gamma(z)) = \val(\Gamma_u(z)) \cdot \val(\Gamma_w(z))$.
By multiplying each of the functions $u$ and $w$ by $-1$, and by using Statement~(A.1),
we deduce that for every $k \in \{1,2,\ldots,K-1\}$ there is%
\footnote{In fact, the games $(\Gamma_k)_{k=1}^{K-1}$ are taken from four games.}
 a polynomial game $\Gamma_k$ such that
$\val(\Gamma_k(z)) = \val(\Gamma_u(z)) \cdot \val(\Gamma_w(z))$ for every $z \in (h_k,h_{k+1})$.
This, together with an inductive argument and Statement~(A.6), implies that Statement~(A.7) holds.
\end{proof}

\bigskip

We now have the tools to prove that every continuous piecewise rational function is in $\calV$.

\begin{proposition}
\label{prop:4}
Let $u : \dR \to \dR$ be a continuous function that is
piecewise rational. Then $u \in \mathcal{V}$.
\end{proposition}

\begin{proof}
Since the function $u$ is piecewise rational,
there are $-\infty = h_1 < h_2 < \cdots < h_K = \infty$ and rational functions $\left(\tfrac{Q_k}{R_k}\right)_{k=0}^{K-1}$
such that $u = \tfrac{Q_k}{R_k}$ on the interval $(h_k,h_{k+1})$, for every $k=1,2,\ldots,K-1$.
Assume w.l.o.g.~that for every $k=1,2,\ldots,K-1$ the polynomial $R_k$ is positive on $(h_k,h_{k+1})$,
and denote by $\ep_k := \inf\{R_k(z) \colon z \in (h_k,h_{k+1})\}$ the minimum of $R_k$ in this interval.

We argue that $\ep_k > 0$ for every $k=0,1,\ldots,K-1$.
For $k=2,\ldots,K-2$ we have $\ep_k > 0$ since the interval $(h_k,h_{k+1})$ is bounded and since $R_k > 0$ on $[h_k,h_{k+1}]$.
For $k=1$ we have $\ep_k > 0$ since $\lim_{z \to -\infty} R_1(z) > 0$ and since $R_1(z) > 0$ on $(-\infty,h_1]$.
For analogous reasons, $\ep_{K-1} > 0$.

By Statement~(A.4) of Proposition~\ref{prop:minus} we have $\max\{R_k,\ep_k\} \in \calV$ for every $k=1,\ldots,K-1$.
By Statements~(A.5) and~(A.7) of Proposition~\ref{prop:minus} and since $\ep_k > 0$ we have $\frac{Q_k}{\max\{R_k,\ep_k\}} \in \calV$ for every $k=1,\ldots,K-1$.
By iterative use of Statement~(A.6) of Proposition~\ref{prop:minus},
we have
\[ \sum_{k=1}^{K-1} \textbf{1}_{z \in (h_k,h_{k+1}]} \frac{Q_k}{\max\{R_k,\ep_k\}} \in \calV. \]
However, this function is equal to $u$, and the result follows.
\end{proof}

\section{Discussion}
\label{sec:discussion}

A linear game is a polynomial game where the payoff in each entry is a linear function of the parameter $Z$.
The set of value functions of linear games is a strict subset of $\mathcal{V}$.
Indeed, in a linear game the maximal payoff is linear in $Z$,
hence the value function is bounded by a linear function of $Z$.
In particular, the function $Z^2$ cannot be the value function of any linear game.

When the domain of the parameter $Z$ is a compact interval $I$ rather than $\dR$,
the set $\calV_I$ of value functions of linear games coincides with the restriction of $\calV$ to $I$;
that is, the set $\calV_I$ is the set of all continuous and piecewise rational functions from $I$ to $\dR$.
The only parts in our proof that do not carry over to linear games are the proofs of Conditions~(A.6) and~(A.7).
To see that Condition~(A.6) holds for linear games, we note that one can prove by induction on $k$ that
$Z^k \in \calV_I$, for every $k \in \dN$.
Instead of proving Condition~(A.7), one can prove by induction on the degree of $Q$ that
for every positive real number $\ep > 0$ and every two polynomials $Q$ and $R$,
the function $\frac{Q}{\max\{R,\ep\}} \in \calV_I$,
which is the property that is needed in the proof of Proposition~\ref{prop:4}.


One example of a linear game that is defined over a compact interval
is a Bayesian game with two states of nature $s_0$ and $s_1$,
where the parameter $Z$ is the prior probability that the state is $s_0$.
Thus, our result implies that the set of all value functions of Bayesian games with two states of nature
is the set of all continuous piecewise rational functions defined on $[0,1]$.

A natural question that arises from our study is what happens when the payoffs depend on two (or more) parameters, say $Z$ and $W$.
In the context of Bayesian games,
this will be equivalent to Bayesian games with more than two states of nature.
In that case, the value function is a function that assigns a real number to every $(Z,W) \in \dR^2$.
Some parts of our arguments extend to this case.
The difficult step is to prove that Condition~(A.6) holds.
This condition turns out to be related to the Pierce-Birkhoff Conjecture,
which asks whether every continuous piecewise polynomial function in $\dR^d$ is the maximum of finitely many minima of finitely many polynomials.
We leave the characterization of the set of value functions of polynomial games with several parameters for future research.

\end{document}